\documentclass[11pt]{article}

\usepackage{graphicx}

\usepackage{calc}

\usepackage{amsmath}
\usepackage{amsxtra}
\usepackage{amssymb}
\usepackage{amsfonts}
\usepackage{amsthm}

\usepackage[sans]{dsfont}

\usepackage{color}
\definecolor{lightgrey}{rgb}{0.7,0.7,0.7}

\usepackage{hyperref}

\usepackage{textcomp}
\usepackage{mathrsfs}

\newcommand{\defm}[1]{{\it #1}}
\newcommand{\Z}{\mathds{Z}}
\newcommand{\R}{\mathds{R}}
\newcommand{\N}{\mathds{N}}

\newcommand{\eps}{\epsilon}
\DeclareMathOperator{\graph}{graph}
\DeclareMathOperator{\BV}{BV}
\newcommand{\subeq}[2]{\mathord{\underbrace{\mathop{#1}}_{#2}}}

\newcommand{\topref}[2]{\overset{\text{\myeqref{#1}}}{#2}}
\newcommand{\myeqref}[1]{\eqref{#1}}

\newcommand{\dnconv}{\searrow}

\newtheorem{lemma}{Lemma}
\newtheorem{proposition}{Proposition}
\newtheorem{corollary}{Corollary}
\newtheorem{theorem}{Theorem}
\newtheorem{definition}{Definition}
\theoremstyle{remark}
\newtheorem{remark}{Remark}
\newtheorem{example}{Example}

\title{H\"older continuity and differentiability on converging subsequences\footnote{This material is based upon work partially supported by the
National Science Foundation under Grant No.\ NSF DMS-1054115 and by a Sloan Foundation Research Fellowship.}}
\author{Volker Elling}
\date{}

\begin{document}

\maketitle

\begin{abstract}
    It is shown that an arbitrary function from $D\subset \R^n$ to $\R^m$ will become $C^{0,\alpha}$-continuous in almost every $x\in D$
    after restriction to a certain subset with limit point $x$.
    For $n\geq m$ differentiability can be obtained.
    Examples show the H\"older exponent $\alpha=\min\{1,\frac{n}{m}\}$ is optimal.
\end{abstract}

\section{Motivation}

Many applications require considering solutions of differential equations that do not have sufficient
regularity to interpret the derivatives in the standard sense. Instead, ``weak solutions''
are defined, for example in the distributional sense or as viscosity solutions.
Since differential calculus is more powerful than awkward manipulation of weak solutions, 
it is desirable to be able to use the differential equations in their original sense at least for some purposes.

For example, consider a system of conservation laws (subscripts are derivatives)
$$ (f(V)-\xi V)_\xi=-V $$ 
where $f:\R^m\rightarrow\R^m$ is smooth and the $\xi$ derivative is interpreted in the distributional sense. 
Assume strict hyperbolicity: $f:\R^m\rightarrow\R^m$ is smooth, and for all $V\in\R^m$ the Jacobian $f'(V)$ has distinct real eigenvalues
$\lambda^1(V)<...<\lambda^m(V)$, with corresponding eigenvectors $r^1(V),...,r^m(V)$ chosen to depend smoothly on $V$. 
Assume $\lambda^k$ is \defm{linearly degenerate}: $\lambda^k_V(V)r_k(V)=0$.

Assume the system, as well as 
$$ \lambda^k(V(\xi))=\xi, $$
are satisfied on a positive-length interval $J$ of $\xi$, hence
$$ \lambda^k_V(V(\xi))V_\xi(\xi) = 1 . $$
If $V$ is a smooth solution, then $0=(f_V(V)-\xi I)V_\xi=(f_V(V)-\lambda^k(V)I)V_\xi$, but then $V_\xi$ is a multiple of $r_k(V)$, so
by linear degeneracy
$$\lambda^k_V(V)V_\xi=0$$
--- contradiction.

This argument also works for $V\in\BV(J)$, since such functions are almost everywhere differentiable.
But $V\in L^\infty(J)$ or even $V\in C(J)$ need not be differentiable anywhere (see \cite{hunt-nowhere-diffable} and references therein).
But we need less: the argument would work if we could find an $x\in J$ and a sequence $(x_n)\rightarrow x$
so that $V_{|C}$ with $C=\{x\}\cup\{x_n:n\in\N\}$ is differentiable in $x$. 

This is not generally true for an arbitrary $V:\R^n\rightarrow\R^m$ with $m>n$ 
(as Example \ref{ex:exponent-sharp} shows); the most one can guarantee is H\"older-regularity with exponent $\frac{n}{m}<1$
on $C$. Nevertheless, we will show the result is true for $m=n=1$, so a somewhat more elaborate argument works
in \cite[Section 16 Lemma 2]{elling-roberts}.

\section{Related work}

Differentiability after restriction --- along with other properties such as continuity and monotonicity ---
has been considered occasionally throughout the literature. 
The book \cite{bruckner-diffability-book} gives a relatively recent overview. 
\cite{bruckner-ceder-weiss} and \cite{ceder-diffable-roads} provide differentiability after restriction,
but allowing $\pm\infty$ as derivatives which is precisely what we need to avoid for our purposes.
Closest to our interest is \cite{brown-tatra}: he states (page 4, problem 1(2.)) that for any $P\subset[0,1]$
with positive measure (meaning Lebesgue outer measure throughout this article)
and any $f:P\rightarrow\R$ there exists a $Q\subset P$, bilaterally dense-in-itself, so that $f_{|Q}=g_{|Q}$ for $g\in C^1[0,1]$.
This is close to our Corollary \ref{cor:diffability}.

In addition, we also obtain Lipschitz and H\"older estimates in higher dimensions, a topic that does not appear to have been
studied. These estimates are needed to extend \cite{elling-roberts} to certain 
systems of hyperbolic conservation laws that may have linearly degenerate eigenvalues with eigenspaces of dimension $>1$,
a case that cannot be handled by scalar 1-d results alone.
Finally, we provide various counterexamples showing in particular that none of the H\"older exponents can be improved.

\section{Main result}

We denote by $\mu_k(A)$ the Lebesgue (outer) measure of $A\subset\R^k$. $B_r(x)$ is the open ball of radius $r$ around $x$,
$\overline B_r(x)$ the closed ball.

\begin{definition}
    Let $\alpha\in(0,1]$, $A\subset\R^n$.

    We say $f:A\rightarrow\R^m$ is \defm{uniformly} $C^{0,\alpha}$ with constant\footnote{need not be the infimum} $M\in[0,\infty)$ if 
    $$ \forall x,x'\in A:|f(x)-f(x')|\leq M|x-x'|^\alpha \quad.$$
    We say $f$ is \defm{locally} $C^{0,\alpha}$ in $x\in A$ with constant $M\in[0,\infty)$ if 
    there is an $r>0$ so that
    $$ \forall x'\in \overline B_r(x)\cap A: |f(x)-f(x')| \leq M |x-x'|^\alpha \quad.$$
    Or equivalently:
    $$ \forall \rho\in(0,r]:f(\overline B_\rho(x))\subset \overline B_{M\rho^\alpha}(f(x))$$
\end{definition}

\begin{lemma}
    \label{lemma:hoelder-measure-change-balls}
    Consider $m\geq n$, $A\subset\R^m$, and let $f:A\rightarrow\R^n$ be locally $C^{0,\beta}$ in $x\in A$
    with exponent $\beta=\frac{m}{n}$, with $r$ and constant $M<\infty$ as in the definition. Then
    there is a constant $C<\infty$ depending only on $n,m$ so that 
    $$ \mu_n\big(f(\overline B_r (x)\cap A)\big) \leq C M^n \mu_m(\overline B_r (x)) \quad.$$
\end{lemma}
\begin{proof}
    By local $C^{0,\beta}$ regularity, 
    \begin{alignat}{5}
        \mu_n\big(f(\overline B_r (x)\cap A)\big) 
        &\leq \mu_n\big(\overline B_{Mr ^\beta}(f(x))\big) 
        \\&= (Mr ^\beta)^n\mu_n(\overline B_1)
        \\&= M^nr ^m\mu_n(\overline B_1)
        \\&= M^n\subeq{\frac{\mu_n(\overline B_1)}{\mu_m(\overline B_1)}}{=:C}r ^m\mu_m(\overline B_1)
        \\&= CM^n\mu_m(\overline B_r (x))
        \notag\end{alignat}
\end{proof}

\begin{proposition}
    \label{prop:hoelder-measure-change}%
    Let $m\geq n$.
    Let $A\subset\R^m$ bounded. 
    Consider an $f:A\rightarrow \R^n$ that is locally $C^{0,\beta}$ in every $x\in A$, with exponent $\beta=\frac{n}{m}$ and constant $M<\infty$ 
    independent of $x$. Then
    $$ \mu_n(f(A))\leq C M^n\mu_m(A) $$
    where $C<\infty$ depends only on $n,m$.
\end{proposition}
\begin{proof}
    Let $U\supset A$ be open 
    so that $\mu_m(U)\leq 2\mu_m(A)$. 
    
    Let $\eps>0$ be arbitrary. 

    In each $x\in A$, $f$ being locally $C^{0,\beta}$ implies there is an $r(x)\in(0,\infty)$ so that 
    $$\overline B_{(2+\eps)r(x)}(x)\subset U$$
    and 
    $$\forall y\in\overline B_{(2+\eps)r(x)}(x)\cap A :\ |f(x)-f(y)|\leq M|x-y|^\beta.$$
    (Note that $M,\beta$ are independent of $x$.) 
    
    We choose a sequence $(x_k)\subset A$ for $k=1,2,3,...$ (possibly finite) as follows: 
    consider the balls that do not meet the previously chosen ones:
    $$ X_k = \{ y\in A ~|~ \forall j<k:\overline B_{r(x_j)}(x_j)\cap\overline B_{r(y)}(y)=\emptyset\} \quad.$$
    Terminate if $X_k=\emptyset$, otherwise choose $x_k\in X_k$ with near-maximal radius:
    \[{ r(x_k)\geq(1+\epsilon)^{-1}\sup_{x\in X_k}r(x) \label{eq:nearmax} }\] 
    Claim: 
    \[{ A\subset\bigcup_k\overline B_{(2+\epsilon)r(x_k)}(x_k)\quad. \label{eq:AinB} }\] 
    Assume not. Then there is a $y\in A$ that is not in the union. Let $m$ be minimal so that 
    $$r(x_m)<(1+\eps)^{-1}r(y)\quad.$$
    (Such an $m$ exists because the chosen balls are pairwise disjoint and contained in $U$ which has finite measure.)
    Then $y\not\in X_m$ because otherwise the choice of $x_m$ would require
    $$r(x_m)\topref{eq:nearmax}{\geq}(1+\eps)^{-1}\sup_{x\in X_m}r(x)\geq(1+\eps)^{-1}r(y)>r(x_m).$$
    Thus $\overline B_{r(y)}(y)$ meets at least one $\overline B_{r(x_j)}(x_j)$ with $j<m$ (hence $r(y)\leq(1+\eps)r(x_j)$, by definition of $m$),
    and therefore
    $$y\in\overline B_{r(y)+r(x_j)}(x_j)\subset\overline B_{(2+\eps)r(x_j)}(x_j)$$
    --- contradiction.

    Now
    \begin{alignat}{1}
        \mu_n(f(A))
        &\topref{eq:AinB}{=} \mu_n\Big(f\big[\bigcup_k\overline B_{(2+\epsilon)r(x_k)}(x_k)\cap A\big]\Big)
        \\&= \mu_n\Big(\bigcup_k f\big[ \overline B_{(2+\epsilon)r(x_k)}(x_k)\cap A\big]\Big)
        \\&\leq \sum_k\mu_n\big(f[\overline B_{(2+\epsilon)r(x_k)}(x_k)\cap A]\big) 
        \intertext{(use Lemma \ref{lemma:hoelder-measure-change-balls})}
        &\leq C(n,m)M^n \sum_k\mu_m(\overline B_{(2+\epsilon)r(x_k)}(x_k))
        \\&= C(n,m)M^n (2+\eps)^m \sum_k\mu_m(\overline B_{r(x_k)}(x_k))
        \intertext{($\overline B_{r(x_k)}(x_k)$ pairwise disjoint and closed, hence measurable, and countable family)}
        &= C(n,m)M^n (2+\eps)^m \mu_m\big(\bigcup_k\overline B_{r(x_k)}(x_k)\big)
        \intertext{($B_{r(x_k)}(x_k)\subset U$ by choice of $r(x_k)$)}
        &\leq C(n,m)M^n (2+\eps)^m \mu_m(U)
        \\&\leq 2C(n,m)(2+\eps)^m  M^n \mu_m(A)
        \notag\end{alignat}
\end{proof}

\begin{theorem}
    \label{th:hoelder-countable}%
    Let $m\geq n$.
    Consider any $D\subset\R^n$ and $f:D\rightarrow\R^m$.
    Set $\alpha=\frac{n}{m}$.
    For almost every $x\in D$ there is a sequence $(x_k)\subset D\backslash\{x\}$ converging to $x$ with
    $$ \limsup_{k\rightarrow\infty} \frac{|f(x)-f(x_k)|}{|x-x_k|^\alpha} < \infty $$
\end{theorem}
\begin{proof}
    The proof has the flavor of an ``inverted Sard lemma''. A sketch of the case $n=m=1$ was suggested by
    Stefano Bianchini \cite{bianchini-ctlip}.

    \noindent\parbox{\textwidth-2.4in}{%
        
        The result is clearly void if $\mu_n(D)=0$, so consider $\mu_n(D)>0$. 

        Define 
        $$ \Omega(x,r) := \inf_{x'\in (D\cap\overline B_r(x))\backslash\{x\}}\frac{|f(x)-f(x')|}{|x-x'|^\alpha} \quad;$$ 
        $\Omega(x,r)$ is decreasing in $r$.
        The sequence in the statement exists for $x$ if and only if $\Omega(x,r)$ is bounded as $r\dnconv 0$.
        Assume the theorem is false, then there is an $E\subset D$ with $\mu_n(E)>0$ so that
        $\Omega(x,r)\rightarrow\infty$ as $r\downarrow 0$ for any $x\in E$.
    }\hskip0.2in\parbox{2.2in}{\begin{picture}(0,0)%
\includegraphics{funnel.pstex}%
\end{picture}%
\setlength{\unitlength}{3947sp}%
\begingroup\makeatletter\ifx\SetFigFont\undefined%
\gdef\SetFigFont#1#2#3#4#5{%
  \reset@font\fontsize{#1}{#2pt}%
  \fontfamily{#3}\fontseries{#4}\fontshape{#5}%
  \selectfont}%
\fi\endgroup%
\begin{picture}(2610,1749)(-11,-898)
\put(108,-23){\makebox(0,0)[lb]{\smash{{\SetFigFont{6}{7.2}{\rmdefault}{\mddefault}{\updefault}{\color[rgb]{0,0,0}$\graph f$}%
}}}}
\put(1079,-77){\makebox(0,0)[lb]{\smash{{\SetFigFont{6}{7.2}{\rmdefault}{\mddefault}{\updefault}{\color[rgb]{0,0,0}$x$}%
}}}}
\put(1455,192){\rotatebox{20.0}{\makebox(0,0)[lb]{\smash{{\SetFigFont{6}{7.2}{\rmdefault}{\mddefault}{\updefault}{\color[rgb]{0,0,0}$x'\mapsto x+\Omega(x,|x-x'|)$}%
}}}}}
\end{picture}%
\\$n=m=1$: if there are ``not enough'' $C^{0,1}$ subsequences,
        then inverses of $f$ with arbitrarily small $C^{0,1}$ norm yield a contradiction.}
    
    By definition of $\Omega$,
    \begin{alignat}{1}&
        \forall x'\in E :\ |f(x')-f(x)|\geq\Omega\big(x,|x-x'|\big)|x-x'|^\alpha; \label{eq:funnel}
    \end{alignat}
    this inequality also holds when replacing $E$ with any $\tilde E\subset E$ (\emph{without} changing $\Omega$ to 
    be defined using $E$ or $\tilde E$ instead of $D$). 

    \paragraph{Boundedness}
    
    $$ E = \biguplus_{k\in\Z} f^{-1}[k,k+1)\cap E ;$$
    since $E$ has positive (outer) measure, at least one $f^{-1}[k,k+1)\cap E$ has positive (outer) measure. 
    Call it $\tilde E$.
    $f(\tilde E)\subset[k,k+1)$ is bounded.

    \paragraph{Injectivity}

    For $x\in \tilde E$ set
    $$ \rho(x) := \sup\{r\in(0,1]~|~\Omega(x,r)>0\}\quad;$$
    $\rho(x)>0$ since $\Omega(x,r)\rightarrow\infty$ as $r\downarrow 0$. Then
    \begin{alignat}{1}&
        \forall x'\in \tilde E\cap B_{\rho(x)}(x):\ x\neq x' \ \Rightarrow\ f(x)\neq f(x')\quad, \label{eq:partinj}
    \end{alignat}
    because $0<|x-x'|<\rho(x)$ yields $\Omega(x,|x'-x|)>0$ in \eqref{eq:funnel}.
    
    $$\tilde E=\biguplus_{k\in\Z}\tilde E_k\quad,\quad \tilde E_k:=\big\{x\in \tilde E~\big|~\rho(x)\in[2^k,2^{k+1})\big\}\quad,$$ 
    so at least one $\tilde E_k$
    must have positive outer measure as well. Call it $\hat E$ and set $\rho:=2^k$ 
    (so $\rho\leq\rho(x)$ for all $x\in\hat E$).
    Now
    $$ \hat E = \biguplus_{k\in\Z} \big(k\rho,(k+1)\rho\big]\cap\hat E ,$$
    so at least one of the sets in the union must have positive outer measure as well. 
    Call it $E'$.
    For any $x\in E'$ we have $E'\subset B_{\rho}(x)\subset B_{\rho(x)}(x)$, 
    so by \eqref{eq:partinj} $f_{|E'}$ is \emph{injective}.

    \paragraph{Zero measure contradiction}
    
    Let $g$ be the inverse of $f:E'\rightarrow f(E')$. 
    For any $x,x'\in E'$ we have by construction of $E'$ that $\Omega(x,|x-x'|)\geq\Omega(x,\rho)>0$.
    Therefore, \eqref{eq:funnel} in the form
    \begin{alignat}{5}
        \forall x,x'\in E': |x-x'| &\leq \subeq{\Omega(x,|x-x'|)^{-1/\alpha}}{\leq\Omega(x,\rho)^{-1/\alpha}}~|f(x')-f(x)|^{1/\alpha} 
        \label{eq:funnel2}
    \end{alignat}
    implies $x'\rightarrow x$ as $f(x')\rightarrow f(x)$. In particular $\Omega(x,|x-x'|)^{-1/\alpha} \downarrow 0$
    as $f(x')\rightarrow f(x)$.     
    Therefore \eqref{eq:funnel2} shows 
    that for $\delta>0$ arbitrarily small
    $g$ will be locally $C^{0,1/\alpha}=C^{0,\frac{m}{n}}$ with constant $\delta>0$ in any $y\in f(E')$.
    By Proposition \ref{prop:hoelder-measure-change} that means
    $$ \mu_n(E') = \mu_n\big(g(f(E'))\big) \leq C(n,m)\delta^n\mu_m(f(E')) .$$
    $\mu_m(f(E'))\leq\mu_m(f(\tilde E))<\infty$, so $\delta\downarrow 0$ yields $\mu_n(E')=0$ --- contradiction!    
\end{proof}

\section{Sharpness of H\"older exponents}

\newcommand{\cRplus}{\overline\R_+}
\begin{example}
    \label{ex:exponent-sharp}%
    The H\"older exponent in Theorem \ref{th:hoelder-countable} cannot be improved: 
    for any $\alpha>\frac{m}{n}$ there is a function $f:(\cRplus)^n\rightarrow(\cRplus)^m$
    so that for any sequence $(x_k)\subset(\cRplus)^n$ with limit point $x\in(\cRplus)^n$, $f$ is not $C^{0,\alpha}$ in $x$. 
\end{example}
\begin{proof}
    Let $t\geq 2$. 
    Consider the $t$-ary representation 
    \begin{alignat*}{5} x_j = \sum_{k=-\infty}^\infty \sum_{i=0}^{m-1} t^{-mk-i} x_{kji} \end{alignat*} 
    Every $x_i\in\cRplus$ can be represented in this way for at most two sequences of digits $x_{kji}\in\{0,...,t-1\}$; 
    we choose the unique one that does not end in an infinite sequence of $t-1$. Set 
    \begin{alignat*}{5} f_i(x) = \sum_{k=-\infty}^\infty \sum_{j=0}^{n-1} (t+1)^{-nk-j} x_{kji} \end{alignat*} 
    Consider $x,y\in(\cRplus)^n$, $x\neq y$. Let $KJI$ be the least index $kji$ (in lexicographic order) so that $x_{kji}\neq y_{kji}$. Then 
    (use the $\ell^1$ vector norm $|\cdot|_1$)
    \begin{alignat*}{5} |x-y|_1 
        &\leq \sum_{j=0}^{n-1} \sum_{k=K}^\infty \sum_{i=0}^{m-1} t^{-mk-i} \subeq{|x_{kji}-y_{kji}|}{\leq t-1} 
        \\&= n(t-1) \sum_{\ell=mK}^\infty t^{-\ell} 
        \\&= \frac{n(t-1)t^{-mK}}{1-t^{-1}}
        \\&= nt^{1-mK}
    \end{alignat*} 
    while
    \begin{alignat*}{5} 
        |f(x)-f(y)|_1
        &\geq
        |f_I(x)-f_I(y)|
        \intertext{(next, $kj\geq KJ$ is lexicographic order)}
        &=
        \Big| \sum_{kj\geq KJ}^\infty (x_{kjI}-y_{kjI}) (t+1)^{-nk-j} \Big|
        \\&\geq
        \subeq{|x_{KJI}-y_{KJI}|}{\geq 1}
        (t+1)^{-nK-J}
        -
        \sum_{kj>KJ}
        \subeq{|x_{kjI}-y_{kjI}|}{\leq t-1}
        (t+1)^{-nk-j}
        \\&\geq
        (t+1)^{-nK-J}
        -
        (t-1)\sum_{\ell>nK+J}(t+1)^{-\ell}
        \\&=
        (t+1)^{-nK-J}
        -
        (t-1)(t+1)^{-nK-J} \frac{(t+1)^{-1}}{1-(t+1)^{-1}} 
        \\&=
        (t+1)^{-nK-J}[
        1-
        (t-1) \frac{1}{t} ]
        \\&=
        (t+1)^{-nK-J}
        \frac1t
        \\&\geq
        (t+1)^{-nK-J-1}
        \\&=
        t^{(-nK-J-1)\log(t+1)/\log t}
    \end{alignat*} 
    Let $y\rightarrow x$, so that $|x-y|_1\searrow 0$, i.e.\ $K\rightarrow\infty$, then
    \begin{alignat*}{5} \frac{|f(x)-f(y)|_1}{|x-y|_1^\alpha} \sim t^{(\alpha m-n\log(t+1)/\log t)K} \end{alignat*} 
    By assumption $\alpha m-n>0$, 
    hence $\alpha m-n\log(t+1)/\log t>0$ for $t$ sufficiently large
    so that $t^{\alpha m-n\log(t+1)/\log t}>1$. Then the expression converges to infinity.
    Hence $f$ is not $C^\alpha$. 
\end{proof}

\begin{remark}
    The proof requires $t\rightarrow\infty$ as $\alpha\dnconv\frac{n}{m}$. 
    It is also possible to obtain a single $f$ that provides a counterexample for all $\alpha>\frac{n}{m}$, 
    by choosing an increasing $t$ for deeper digits, but we prefer to omit the tedious construction. 
\end{remark}

\section{$C^{0,\alpha}$ on entire sequence}

In some applications it is convenient to have a version of Theorem \ref{th:hoelder-countable}
with uniform $C^{0,\alpha}$ estimates rather than $C^{0,\alpha}$ in $x$ alone:
\begin{theorem}
    \label{th:uniform-hoelder}%
    In the setting of Theorem \ref{th:hoelder-countable} 
    there is a $C\subset D$ with limit point $x$ so that
    $f$ is uniformly $C^{0,\alpha}$ on $C$.
\end{theorem}
\begin{proof}
    Consider $x$ and $(x_n)$ as provided by Theorem \ref{th:hoelder-countable}. 
    We desire a subsequence $(x_n')$ so that
    $$ \sup_{x'_n\neq x'_m} \frac{|f(x_m')-f(x_n')|}{|x_m'-x_n'|^\alpha} < \infty $$
    as well. This is accomplished as follows: we already know
    $$ \forall n\in\N:\ |f(x)-f(x_n)| \leq M|x-x_n|^\alpha .$$
    We select $x'_n$ inductively. Take $x'_1=x_1$.
    For $m=2,3,4,...$ choose $x'_m\in\{x_n:n\in\N\}$ so that 
    $$ |x-x'_m| \leq \delta_m \min_{k<m} |x-x'_k| $$
    for some $\delta_m<1$ to be determined (in particular the $x'_n$ must be pairwise different).
    Then for any $k<m$, 
    $$ |x'_k-x'_m| \geq |x'_k-x| - |x-x'_m| \geq |x'_k-x| (1-\delta_m) \quad\Rightarrow\quad|x'_k-x|\leq(1-\delta_m)^{-1}|x'_k-x'_m|,$$
    so
    $$ |x-x'_m| \leq |x-x'_k| + |x'_k-x'_m| \leq (1+(1-\delta_m)^{-1}) |x'_k-x'_m| $$
    and both combined
    \begin{alignat}{1}&
        |f(x'_m)-f(x'_k)| 
        \leq |f(x'_m)-f(x)| + |f(x)-f(x'_k)| 
        \leq M|x-x'_m|^\alpha + M|x-x'_k|^\alpha
        \notag\\&\leq M \Big( \big(1+(1-\delta_m)^{-1}\big)^\alpha + (1-\delta_m)^{-\alpha} \Big) |x'_k-x'_m|^\alpha  
        \notag\end{alignat}
    For $\delta_m\downarrow 0$, the factor after $M$ clearly converges to $2^\alpha+1$.
    Hence we may select each $\delta_m>0$ successively to obtain, for any $\epsilon>0$, a constant $(2^\alpha+1+\epsilon)M$ for a uniform $C^{0,\alpha}$ 
    estimate on $C=\{x\}\cup\{x'_n:n\in\N\}$.
\end{proof}

\section{Differentiability}

\begin{corollary}
    \label{cor:diffability}%
    Consider $n=m=1$, any $D\subset\R$ and any $f:D\rightarrow\R$.
    For almost every $x\in D$ there is a sequence $(x_k)\subset D$ so that with $C=\{x\}\cup\{x_k:k\in\N\}$,
    $$\text{$(f_{|C})'(x)$ exists and is finite}.$$
\end{corollary}
\begin{proof}
    Here $\alpha=\frac{n}{m}=1$.
    Choose a sequence according to Theorem \ref{th:hoelder-countable}.
    Since the difference quotients are contained in a compact interval $[-M,M]$, we can choose a subsequence $(x'_k)$ so that they
    converge.
\end{proof}

\bibliographystyle{amsalpha}
\bibliography{../../pmeyer/elling.bib}

\end{document}